\documentclass[12pt,reqno]{amsart}
\usepackage[utf8]{inputenc}
\usepackage[english]{babel}
\usepackage{comment}
\usepackage{amsmath,amssymb,amsfonts,amsbsy,amsthm,amscd,latexsym,graphicx}
\usepackage{empheq,textcomp}
\usepackage{xcolor}
\usepackage{hyperref}
\usepackage{enumerate}

\theoremstyle{definition}
\newtheorem{theorem}{Theorem} [section]
\newtheorem{corollary}[theorem]{Corollary}
\newtheorem{lemma}[theorem]{Lemma}

\newtheorem{definition}[theorem]{Definition}

\newtheorem{remark}[theorem]{Remark}

\numberwithin{equation}{section}

\AtEndDocument{
  \par
  \bigskip
  \begin{tabular}{@{}l@{}}%
    \text{Department of Mathematics, University of Oregon, Eugene, OR 97403-1222, USA}\\
    \textit{E-mail address}: \texttt{mbownik@uoregon.edu}\\ 
   
    \textit{E-mail address}: \texttt{putingyu@uoregon.edu}\\
  \end{tabular}}

\newcommand{\C}{\mathbb{C}}

\newcommand{\Fc}{{\mathcal{F}}}
\newcommand{\Gc}{{\mathcal{G}}}
\newcommand{\Nc}{{\mathcal{N}}}
\newcommand{\Kc}{{\mathcal{K}}}
\newcommand{\N}{\mathbb{N}}
\newcommand{\R}{\mathbb{R}}

\newcommand{\Wc}{{\mathcal{W}}}

\newcommand{\Eq}{\, = \,}

\newcommand{\Le}{\, \le \,}

\newcommand{\qeddef}{{\quad $\diamondsuit$}}

\newcommand{\bigabs}[1]{\bigl|\,#1\,\bigr|}

\newcommand{\ip}[2]{\langle\,#1,#2\,\rangle}
\newcommand{\bigip}[2]{\bigl\langle \,#1, \, #2 \,\bigr\rangle}

\newcommand{\norm}[1]{\|\,#1\,\|}
\newcommand{\bignorm}[1]{\bigl\|\,#1\,\bigr\|}
\newcommand{\Bignorm}[1]{\Bigl\|\,#1\,\Bigr\|}

\newcommand{\bigparen}[1]{\bigl(\,#1\,\bigr)}
\newcommand{\Bigparen}[1]{\Bigl(\,#1\,\Bigr)}

\newcommand{\set}[1]{\{#1\}}
\newcommand{\bigset}[1]{\bigl\{\,#1\,\bigr\}}

\newcommand{\clspan}{{\overline{\text{span}}}}

\newcommand{\inN}{_{n\in\N}}
\newcommand{\sumli}{\sum_{n=1}^\infty}

\begin{document}
\title{Uniform discretization of continuous frames}
\author{Marcin Bownik and Pu-Ting Yu}

\thanks{The first author was partially supported by the NSF
grant DMS-2349756.}

\subjclass[2020]{42C15, 42C40}

\keywords{continuous frame, discrete frame, discretization, Gabor systems, wavelets}

\date{\today}

\begin{abstract}
Let $H$ be an infinite-dimensional separable Hilbert space and
let $(X,d,\mu)$ be a metric measure space satisfying the doubling and upper Alhfors regularity conditions at small scale. We prove that every bounded continuous tight frame $\Psi\colon X\rightarrow H$ can be sampled to obtain a frame for $H$, which is uniformly discrete and nearly tight. That is, for every $0<\epsilon<1$, there exist a sampling sequence $\set{x_n}\inN$ in $X$ and $r>0$ such that $\inf_{n\neq m}d(x_n,x_m)\geq r$ and $\set{\Psi(x_n)}\inN$ is a frame whose ratio of frame bounds is less than $1+\epsilon$.

We apply our main result to show that
for every nonzero function $g$ in $L^2(\R^d)$ there exists a uniformly discrete set $\Lambda$ such that the corresponding Gabor system $\set{e^{2\pi ibx}g(x-a)}_{(a,b)\in \Lambda}$ is a nearly tight frame. We also prove that if $\psi\in L^2(\R)$ satisfies the Calder\'on admissibility condition, then there exists a uniformly discrete set $\Gamma$ such that wavelet system $\set{a^{1/2}\psi(ax-b)}_{(a,b)\in \Gamma}$ is a nearly tight frame. Analogous discretization results for exponential frames and spectral subspaces of elliptic differential operators are presented as well.
\end{abstract}

\maketitle

\section{Introduction}
After being rediscovered by Daubechies, Grossmann and Meyer in \cite{DGM86}, the theory of \emph{frames} has inspired a plethora of research in both applied and theoretical mathematics. Among its numerous noteworthy applications, the connection with \emph{canonical coherent states} in quantum mechanics provides the motivation for the work presented in this paper.
The notion of frames was first proposed by Duffin and Schaefer in their study of non-harmonic series \cite{DS52}. A \emph{frame} for a separable Hilbert space $H$ is a sequence $\set{x_n}_{n\in I}$ in $H$ for which there exist some positive constants $A$ and $B$, called \emph{frame bounds}, such that the following inequality holds: 
\begin{equation}
\label{frame_ineq}
A\,\norm{x}^2\Le \sum_{n\in I} |\ip{x}{x_n}|^2\Le B\,\norm{x}^2 \qquad \text{for all $x\in H$.}
\end{equation}
If $A=B$, then we say $\set{x_n}_{n\in I}$ is a \emph{tight frame} for $H$. In the case when $A=B=1$, $\set{x_n}_{n\in I}$ is referred to as a \emph{Parseval frame} for $H$.
Canonical coherent states are a special family of \emph{coherent states}, which are obtained as the orbits of the unitary action of Weyl-Heisenberg group acting on the ground state of the harmonic oscillator. Beyond this fundamental group action structure, they also exhibit some other frame-like features, such as the resolution of the identity \cite[Chapter 1]{AAG00}, which mirrors the invertibility of the \emph{frame operator} arising from the bounds in (\ref{frame_ineq}). For further connections between frames and canonical coherent states, we refer to \cite{AAG14}, \cite{Bo18a} and \cite{Ka94}.

The notion of resolution of identity, together with the concept of frames was later generalized to \emph{continuous frames} independently by Ali, Antoine and Gazeau in \cite{AAG93}, and by Kaiser in \cite{Ka94}. Let $(X,\mu)$ be a positive measure space. A measurable function $\Psi\colon X\rightarrow H$ is called a continuous frame for $H$ if there exist some constants $A$ and $B$, called \emph{frame bounds}, such that
    \begin{equation}
\label{frame_ineq_continuous}
    A\,\norm{x}^2\Le \int_X\bigabs{\ip{x}{\Psi(t)}}^2\,d\mu(t)\Le B\,\norm{x}^2 \qquad \text{for all $x\in H$.}
    \end{equation}
In the case when $A=B$, we say $\Psi$ is a \emph{continuous tight frame} for $H$. If $A=B=1$, then $\Psi$ is referred to as a \emph{continuous Parseval frame}. Instead of summing over an associated countable index set, the notion of continuous frames adopts integration over underlying measure space $X$. 
The original ``discrete" frames constitute a special family of continuous frames when the associated measure is the counting measure on the index set $I$. 

In a converse direction, Ali, Antoine and Gazeau proposed the discretization problem of whether one can sample a continuous frame to obtain a discrete frame, see \cite[Chapter 16]{AAG00} or \cite[Chapter 17]{AAG14} for an expanded treatment. In the language of signal processing, this question can be reformulated as: How does one sample a continuous signal without losing information? On one hand, a positive answer to this discretization problem establishes a bi-directional theoretical bridge between discrete and continuous frames. On the other hand, the feasibility of continuous frame discretization would significantly facilitate computation when we numerically approximate the integral in (\ref{frame_ineq_continuous}).

    The formulation of the discretization problem of continuous frames can be further refined and broken down into several layers of inquiry: 
    \begin{enumerate}[(i)]\setlength\itemsep{0.5em}
        \item 
   Existence: Is it always possible to obtain a frame by sampling a continuous frame?
   \item Tightness of frame bounds: Equal frame bounds often offers excellent numerical stability and robust reconstruction of elements (for example, see \cite{GKK01}, \cite{HP04}, \cite{Wa03}). 
   If the continuous frame is tight, can one always acquire a frame, through appropriate sampling, with frame bounds that are arbitrarily close to each other?
   \item Distinctness of the sampling sequence: A common property shared by frames and canonical coherent states is overcompleteness. This redundancy often endows frames with great flexibilities in the reconstruction of elements. However, it could potentially result from a large amount of repeated elements, which diminishes the flexibility in the reconstruction of elements and leads to an inefficient representation. Consequently, is it always possible to construct a frame by using distinct sampling of a continuous frame? 
   \end{enumerate}
   
   Regarding the question of existence, Fornasier and Rauhut \cite{FR05} gave a partial answer for a specific family of continuous frames, using the theory of coorbit spaces. The case of coherent states associated with the Poincar\'{e} group in $1+3$ dimensions was solved by Antoine and Hohou\'{e}to \cite{AH03}. The existence problem was solved in full generality by Freeman and Speegle in \cite{FS19}. However, the resulting frames neither possess frame bounds that can be made arbitrarily close nor they avoid the use of duplicate elements. Recently, the first author \cite{Bo24} showed  that one can always sample a continuous frame to obtain a frame with arbitrarily close frame bounds. Nevertheless, the resulting frame may still contain repeated elements. In this paper, we fully resolve the discretization problem in all three aspects, under only mild assumptions on the underlying metric measure space $(X,d,\mu)$ equipped with metric $d$ and measure $\mu$. We assume that $\mu$ is a Borel measure, $\mu(X)=\infty$, and $\mu$ satisfies the doubling condition and upper Ahlfors regularity condition at small scales, see Section 2.

We prove that it is possible to sample a continuous tight frame by a uniformly discrete sequence to obtain a frame whose ratio of frame bounds is arbitrarily close to $1$. Here, we say a sequence $\set{x_n}\inN$ in a metric space $(X,d)$ is \emph{uniformly discrete} if there exists some constant $r>0$ such that $\inf_{n\neq m} d(x_n,x_m)\geq r.$  
Given a frame with frame bounds $A$ and $B$, the \emph{ratio of frame bounds} is defined as the number $B/A.$ 
In the setting when $X$ is the usual Euclidean spaces $\R^d$ equipped with the Lebesgue measure $\mu$ our result reads as follows.      
    \begin{theorem}
    \label{main_result_Lebesgue} 
 Let $\Psi\colon \R^d\rightarrow H$ be a continuous frame with frame bounds $A$ and $B$. Assume that $\norm{\Psi(t)}^2\Le D$ almost everywhere. Then for any $\epsilon>0$, there exists a uniformly discrete sequence $\set{x_n}\inN\subseteq X$ such that $\set{\Psi(x_n)}\inN$ is a frame whose ratio of frame bounds is at most $\frac{B}{A}(1+\epsilon).$
 In particular, if $\Psi$ is a continuous tight frame, then the ratio of frame bounds of $\set{\Psi(x_n)}\inN$ is at most $1+\epsilon$. \qeddef
 \end{theorem}

 A key ingredient in the proof of our main theorem is the selector form of Weaver's $\text{KS}_2$ conjecture, which was established by the first author in \cite{Bo24}. Weaver formulated $\text{KS}_r$ conjecture in \cite{We04}, where he showed that the Kadison-Singer problem has a positive solution if and only if $\text{KS}_r$ conjecture holds for some $r\geq 2$. The $KS_2$ conjecture was confirmed to be true by Marcus, Spielman and Srivastava in \cite{MSS15b}. For further details on the connection between Weaver's $\text{KS}_r$ and the Kadison-Singer problem, as well as related conjectures, we refer to \cite{Bo23}, \cite{Bo24}, \cite{BCMS19}, \cite{CCLV05}, \cite{CFTW06}, \cite{CT06}, \cite{MSS15a}, \cite{MSS15b}, and \cite{We04}.

 Thanks to the solution of the Kadison-Singer problem Nitzan, Olevskii, and Ulanovskii \cite{NOU16} showed the existence of frames of exponentials for unbounded sets. Freeman and Speegle \cite{FS19} extended their result by proving that every bounded continuous frame can be sampled to obtain a discrete frame. Using the selector form of Weaver's $\text{KS}_2$ conjecture \cite{Bo24} we improve upon the result of Freeman and Speegle in two aspects. We show that the sampling sequence can be chosen to be uniformly discrete and the resulting ratio of frame bounds is nearly the same as that of a continuous frame. Related results on sampling sequences near the critical density were obtained by the first author and van Velthoven \cite{MBV1}, \cite{MBV2} using selector techniques.

 We illustrate Theorem \ref{main_result_Lebesgue} for two well known and well studied continuous tight frames. The first application concerns \emph{Gabor systems}. Let $\Lambda\subseteq \R^{2d}$ be a lattice and let $g\in L^2(\R^d)$ be a nonzero function. The Gabor system associated with $g$ and $\Lambda$ is the set $$\Gc(g,\Lambda)\coloneq \set{e^{2\pi ibx}g(x-a)\,|\,(a,b)\in \Lambda}.$$ It is well known that if a window $g$ has a good time-frequency localization, then there exists a uniformly discrete set $\Lambda$, which is in fact a lattice, such that $\Gc(g,\Lambda)$ is a frame for $L^2(\R^d)$, see \cite[Chapter 6]{Gro01}. However, it was an open problem whether the same result holds for an arbitrary nonzero $g\in L^2(\R^d)$. Our next theorem provides a positive answer to this question.

\begin{theorem}\label{uniformly_discrete_Gabor_frame} Let $g\in L^2(\R^d)$ be a nonzero function. For any $\epsilon>0$, there exists a uniformly discrete subset $\Lambda\subseteq \R^{2d}$ such that $\Gc(g,\Lambda)$ is a frame for $L^2(\R^d)$ whose ratio of frame bounds is less than $1+\epsilon.$\qeddef
\end{theorem}

By the Balian-Low Theorem for Gabor frames, see \cite[Corollary 3.7]{CDH99} and \cite[Theorem 1.5]{AFK14}, Theorem \ref{uniformly_discrete_Gabor_frame} is optimal. Indeed, if a window $g$ is well-localized in time and frequency, then no choice of sampling set $\Lambda$ yields a Riesz basis for $L^2(\R^d)$. 
Theorem \ref{uniformly_discrete_Gabor_frame} has an alternative interpretation in the language of frame set.
The \emph{frame set} associated with $g\in L^2(\R^d)$ is the set $$\Fc(g)=\set{\Lambda\subseteq\R^{2d}\,|\,\Gc(g,\Lambda)\text{ is a frame for }L^2(\R^d)}.$$ 
Characterizing the frame set of general functions in $L^2(\R^d)$ remains a challenging problem even for reduced frame sets consisting of lattice generating Gabor frames. Indeed, the reduced frame sets corresponding to lattice $\Lambda \subseteq \R^{2d}$ have been fully characterized only for small classes of windows, see \cite{BKL23}, \cite{Gro13}, \cite{Gro23}. As an immediate consequence of Theorem \ref{uniformly_discrete_Gabor_frame}, we obtain the following corollary regarding frame sets.
\begin{corollary}
    The frame set of every nonzero function $g\in L^2(\R^d)$ contains a uniformly discrete subset.
\end{corollary}

The second application involves the discretization of continuous wavelet transform \cite{Dau}, \cite{HW}. A function $\psi \in L^2(\R)$ is a continuous wavelet with respect to the affine $ax+b$ group $\R \rtimes \R^+$, if it satisfies the Calder\'om admissibility condition in the following: 
 \begin{equation}
 \label{admissibility_condition}
\int_{(-\infty,0)}\frac{|\widehat{\psi}(\xi)|^2}{|\xi|}\,d\xi
= \int_{(0,\infty)}\frac{|\widehat{\psi}(\xi)|^2}{|\xi|}\,d\xi<\infty.
\end{equation}
A discrete wavelet system associated with $\Gamma\subseteq\R\rtimes \R^+$ is the set $$\Wc(\psi,\Gamma)=\set{a^{1/2}\psi(ax-b)\,|\,(b,a)\in \Gamma}.$$
The question of which functions $\psi$ admit a uniformly discrete subset $\Gamma$ such that $\Wc(\psi,\Gamma)$ forms a frame was known only for particular classes of functions \cite{CMR}, \cite{GHHLWW}, \cite{M04}. As a consequence of the variant of Theorem \ref{main_result_Lebesgue} for metric measure spaces, we obtain the following result. 

 \begin{theorem}\label{wave}
 Suppose that $\psi\in L^2(\R)$ satisfies the Calder\'on condition \eqref{admissibility_condition}. For any $\epsilon>0$, there exists a uniformly discrete subset $\Gamma\subseteq\R\rtimes \R^+$ such that $\Wc(\psi,\Gamma)$ is a frame for $L^2(\R)$ whose ratio of frame bounds is less than $1+\epsilon.$
     \qeddef
 \end{theorem}

Theorems \ref{uniformly_discrete_Gabor_frame} and \ref{wave} are merely two examples of the versatility of our main result. Another setting where Theorem \ref{main_result_Lebesgue} applies are the generalized Paley-Wiener spaces associated with elliptic operators, which were recently introduced and studied by Gr\"ochenig and Klotz \cite{Gro23}.

The paper is structured as follows. Notations, definitions, and several required preliminary known results are presented in Section \ref{Preliminaries}. 
Our main result for continuous frames on metric measure spaces, Theorem \ref{frame_uniformly_discrete_elements}, is established in Section \ref{Main_results}. We conclude this paper by presenting  applications of our main result in Section \ref{application}.

\section{Preliminaries}
\label{Preliminaries}
Throughout this paper, $H$ is a separable infinite-dimensional Hilbert space equipped with inner product $\ip{\cdot}{\cdot}.$ 
Given a continuous frame $\Psi$, we always assume the underlying measure space associated with $\Psi$, denoted by  $(X,\Sigma,\mu)$, is a positive and non-atomic measure space. Furthermore, by \cite[Proposition 2.1]{Bo18b}, the support of $\Psi$ (that is, $\set{t\in X\,|\,\Psi(t)\neq 0}$) is a $\sigma$-finite subset of $X$. Hence, without loss of generality, we assume that $X$ itself is $\sigma$-finite. 

We will primarily work under the following three assumptions.
    \begin{enumerate}\setlength\itemsep{0.2em}
    \item [\textup{(I)}] $\mu(X)=\infty$.
    \end{enumerate}
    In the case the space $X$ is equipped with a metric $d$, all balls centered at $x\in X$ with radius $r$, $B_r(x)$, are assumed to have a positive measure for every $r>0$ and every $x\in X.$ We will also make use of the following two assumptions. There exists some $R_A>0$ such that 
    \begin{enumerate}\setlength\itemsep{0.2em}
    \item [\textup{(II)}] (Doubling condition at small scale) for any $0<r\leq R_A$ we have 
    $$\mu(B_{2r}(x))\leq C_{R_A}\mu(B_{r}(x))\quad \text{ for all $x\in X$},$$
    for some universal positive \emph{doubling constant} $C_{R_A}$.
   \item [\textup{(III)}] (Upper Ahlfors regularity condition at small scale) there exist some positive constants $\alpha,\beta>0$ such that for any $x\in X$ and $0< r\leq R_A$ we have $\mu(B_r(x))\Le \alpha r^\gamma.$
       \end{enumerate}

\begin{remark}
We remark that the assumption that all balls have positive measure $\mu(B_r(x))>0$ is made merely for convenience. Indeed, if $0<r<R_A$, then the doubling condition (II) implies that the set
\[
X'=\{x \in X \,|\, \mu(B_r(x))=0 \}
\]
is both closed and open subset of $X$. Thus, given a continuous frame $\Psi: X \to H$, it suffices to perform a discretization of its restriction $\Psi|_{X\setminus X'}$ to the subspace $X \setminus X'$.
\end{remark}

The following lemma is a consequence of the doubling condition and the upper Ahlfors regularity condition. 

\begin{lemma}
\label{partition_cardinality_lemma}
Let $(X,d,\mu)$ be a metric measure space equipped with a positive non-atomic Borel measure $\mu$ satisfying assumptions (II) and (III). Then for any $0<r\leq\frac{R_A}{4}$, there exists a countable partition $\set{X_n}\inN$ of $X$ satisfying the following two statements:
\begin{enumerate}\setlength\itemsep{0.5em}
 \item [\textup{(a)}] $\mu(X_n)\leq \alpha (2r)^\gamma$ for all $n\in\N$,
  \item [\textup{(b)}] If $\set{x_n}\inN\subseteq X$ is a sequence such that $x_n\in X_n$ for all $n\in\N$, then 
  $$\#(B_{2r}(x_m)\cap \{x_n\,|\,n\in\N\})\Le (C_{R_A})^5\quad \text{for all $m\in\N$,}$$
  where $C_{R_A}$ is the doubling constant. 
\end{enumerate} 
\end{lemma}
\begin{proof}
Let $\set{B_{r}(y_n)}\inN$ be a maximal collection of pairwise disjoint balls with radius $r$. The existence of such a countable collection is ensured since $X$ is $\sigma$-finite and all balls have positive measure. Note that this implies that $X=\bigcup_{n\in\N}B_{2r}(y_n)$. Then we let $\set{X_n}\inN$ be a partition of $X$ such that $$B_{r}(y_n)\subseteq X_n\subseteq B_{2r}(y_n) \quad \text{for all }n\in\N.$$ 
Next, let $\set{x_n}\inN$ be a sequence in $X$ such that $x_n\in X_n$ for all $n\in\N.$ Fix $m\in\N$ and define 
$$I_m = \bigset{x_n \,|\,x_n\in B_{2r}(x_m)} \quad\text{and}\quad  M_m=\min\limits_{x_n\in I_m}\mu(X_n).$$
If $ x_n\in B_{2r}(x_m)$, then $d(y_n,x_m) \le d(y_n,x_n)+d(x_n,x_m) <4r$, and hence we have $X_n\subseteq B_{6r}(x_m).$ Since $X_n$ are pairwise disjoint, it follows that  
\begin{equation}
\label{cardinality_estimateI}
\#(I_m)M_m\Le  \mu(B_{6r}(x_m))\Le (C_{R_A})^2\mu(B_{2r}(x_m))
\end{equation}
On the other hand, since $d(x_m,y_n) \le d(x_m,x_n)+d(x_n,y_n)<4r$, we also have $B_{2r}(x_m)\subseteq B_{6r}(y_n)$. Consequently, we obtain 
\begin{equation}
\label{cardinality_estimateII}
\mu(B_{2r}(x_m))\Le \mu(B_{6r}(y_n))\Le (C_{R_A})^3\mu(B_{r}(y_n))\Le (C_{R_A})^3\mu(X_n). 
\end{equation}
Hence, $M_m\geq (C_{R_A})^{-3}\mu(B_{2r}(x_m)).$
It follows by (\ref{cardinality_estimateI}) that 
$$(C_{R_A})^{-3}\#(I_m)\mu(B_{2r}(x_m))\Le \#(I_m)M_m\Le (C_{R_A})^2\mu(B_{2r}(x_m)),$$
which implies $\#I_m\leq (C_{R_A})^5.$ Thus, $\#(B_{2r}(x_m)\cap\{x_n\,|\,n\in\N\})\Le (C_{R_A})^5$ for all $m\in \N$.
\end{proof}

We will need the following elementary result, which can be found in \cite[Lemma 7.2]{Bo24}. Recall that a positive semi-definite operator $T$ defined on $H$ is a linear operator on $H$ satisfying $\ip{Tx}{x}\geq0$ for all nonzero $x\in H.$ For a given closed subspace $M$ of $H$, let $P_M$ be the orthogonal projection onto $M$.

\begin{lemma}
\label{elementary_Hilbert_result}
Let $T$ be a positive semi-definite operator defined on $H$.  Then for any closed subspace $M$ of $H$, there exists some constant $C_M= \bigparen{\norm{P_MTP_M}~\norm{P_{M^\perp}TP_{M^\perp}}}^{1/2}$ such that  
$$-C_M\textbf{I}\Le T-P_MTP_M-P_{M^\perp}TP_{M^\perp}\Le C_M\textbf{I}.$$
\end{lemma}

Following \cite{Bo24}, we define the notion of \emph{binary selectors}. Given any $N\in \N$, by $\set{0,1}^N$ we mean the set of all $N$-tuples, where each component is either $0$ or $1$. For any subset $J$, we use $\#J$ to denote the number of elements in $J.$
\begin{definition}
     Let $I$ be a countable index set and let $\set{J_k}_{k\in J}$ be any partition of $I$ with $\#J_k=2$ for all $k\in J$. Binary selectors of order $1$ are sets $I_0$ and $I_1$ such that $I = I_0 \cup I_1$ and $$\#(J_k\cap I_0)=\#(J_k\cap I_1)=1  \quad \text{for all } k\in J.$$
For $N\geq 2$, we define binary selectors of order $N$ inductively. Suppose that binary selectors of order $N-1$, $I_b$, $b\in \set{0,1}^{N-1}$, are already defined. For each $b\in \set{0,1}^{N-1}$, let $\set{J_{b,k}}_{k\in J}$ be any partition of $I_b$ with $\#J_{b,k}=2$ for all $k\in J$. Then binary selectors of order $N$ are sets $I_{b0}$ and  $I_{b1}$ satisfying $I_b=I_{b0}\cup I_{b1}$ and $\#(I_{b0}\cap J_{b,k})=\#(I_{b1}\cap J_{b,k})=1$ for all $k\in J.$ \qeddef 
\end{definition}

We now state the main tool in the proof of our main theorem. A positive semi-definite operator $T$ on $H$ is said to be \emph{trace-class} if there exists some orthonormal basis $\set{e_n}\inN$ for $H$ such that $$\text{tr}(T)\coloneq\sumli \ip{Te_n}{e_n}<\infty.$$
\begin{theorem}(\cite[Theorem 5.3]{Bo24})
\label{binary_selector_theorem}
    Let $\delta>0$ and let $\set{T_n}\inN$ be a family of positive semi-definite trace-class operators defined on $H$. Assume that $$T\coloneq \sumli T_n\Le \textbf{I} \quad \text{ and }\quad\text{tr}(T_n)\leq \delta\text{ for all }n\in\N$$
    Then there exists some absolute constant $C>0$ such that for any $N\in\N$ with $2^N<\frac{1}{\delta}$ and any intermediate choices of partitions with sets of size $2$ there exist binary selectors $I_b$, $b\in \set{0,1}^N$, that form a partition of $\N$ satisfying \begin{align}
    \Bignorm{2^N\sum_{n\in I_b}T_n-T}\Le C\sqrt{2^N\delta},  \end{align}
 for all $b\in\set{0,1}^N$. 
\end{theorem}
We remark that Theorem \ref{binary_selector_theorem} also applies to any finite collection of positive semi-definite trace-class operators by considering $T_n$ the zero operator for all $n$ large enough. Moreover, fix $\delta>0$ and define the sequence of scalars $(B_n)_{n=0}^\infty$ by 
$$B_0=1, \quad B_{j+1}=B_j+4\sqrt{2^j\delta B_j}+2^{j+1}\delta,~j\geq1.$$
The absolute constant $C$ in Theorem \ref{binary_selector_theorem} is chosen so that 
\[
\sum_{j=0}^{N-1}(B_j-1)\Le C\sqrt{2^N\delta},
\]
for every $N\in \N$ such that $2^N\delta<1$. See \cite[Lemma 5.2]{Bo24} or \cite[Lemma 10.20]{OU16} for a proof. 

Every continuous frame $\Psi$ induces a positive semi-definite bounded operator, called the \emph{frame operator}. The frame operator associated with $\Psi$ is the operator on $H$ is defined by 
\begin{equation}
\label{frame_operator}
S_\Psi(f)=\int_X\ip{f}{\Psi(t)}\Psi(t)\,d\mu(t).
\end{equation}
It is worth noting that Equation (\ref{frame_operator}) should be interpreted weakly in terms of Pettis integral. That is, for all $f\in H$, $S_\Psi(f)$ is the unique vector such that $$\ip{S_\Psi(f)}{g}=\int_X\ip{f}{\Psi(t)}\ip{\Psi(t)}{g}\,d\mu(t)\quad \text{for all }g\in H.$$
Every element $f\in H$ also induces a positive semi-definite rank one operator $T_f$ on $H$ defined by \begin{equation}
\label{rank_one_operator}
T_f(g)\coloneq \ip{g}{f}f.
\end{equation}
For notational convenience, we denote by $S_\Psi$ the frame operator associated with a continuous frame $\Psi.$ Throughout, $T_f$ means the rank one operator defined in Equation (\ref{rank_one_operator}) whenever $f$ is an element in $H$.

\section{Main Result}
\label{Main_results}
We now prove our main result in this section. We begin with the motivation behind assumptions (I)--(III). As we will see in the proof of Lemma \ref{sampling_lemma}, obtaining a frame by sampling a continuous frame $\Psi$ associated with a positive, non-atomic finite measure space is straightforward, as long as $\Psi$ is bounded above $\mu$-almost everywhere. For this reason, our primary focus will be the case $\mu(X)=\infty.$ The upper Ahlfors assumption (III) ensures that, during the sampling process, we can use balls to separate points, while maintaining control over their measures. The doubling assumption (II) allows us to bound the maximum number of sample points contained in certain small ball by Lemma \ref{partition_cardinality_lemma}, which in turn enables us to extract a subsequence that is uniformly discrete. Without the second and third assumptions, we can still obtain frames with arbitrarily close frame bounds consisting of distinct elements, but the resulting sampled frames need not be uniformly discrete

The following lemma shows that we can sample a continuous frame $\Psi$, so that resulting collection can be rescaled to form a frame, while simulateneously ensuring that the sampled elements are not too densely distributed in $X$.

\begin{lemma}
\label{sampling_lemma}
     Let $\Psi\colon X\rightarrow H$ be a continuous frame with an upper frame bound $B$.  Assume that $\norm{\Psi(t)}^2\Le D$ $\mu$-almost everywhere and $\mu(X)=\infty$. Then for any $0<\tilde{\epsilon}, \epsilon<1$ there exists a partition $\set{X_n}\inN$ of $X$ and a sequence $\set{x_n}\inN\subseteq X$ such that the following statements hold:
     \begin{enumerate}\setlength\itemsep{0.5em}
    \item [\textup{(a)}] $\mu(X_n)\leq \tilde{\epsilon}$ for all $n\in \N$,

      \item [\textup{(b)}] $\norm{\Psi(x_n)}^2\Le D$ for all $n\in \N$,
    \item [\textup{(c)}] There exists an increasing sequence of natural numbers $(L_n)\inN$ with $L_1=0$ such that for each $n\in \N$ we have $\set{x_k}_{k=L_n+1}^{L_{n+1}}\subseteq X_n$, 
    \item [\textup{(d)}] There exists a sequence of natural numbers $(\ell_n)\inN$ with $\sum_{k=L_n}^{L_{n+1}-1}2^{-\ell_k}\leq \mu(X_n)$ for all $n\in\N$ such that $$\Bignorm{S_\Psi-\sum_{n\in\N}2^{-\ell_n}T_{\Psi(x_n)}}<\epsilon.$$
    \end{enumerate}
Furthermore, if $X$ is a metric space equipped with a metric $d$ and $(X,\mu)$ satisfies assumptions (II) and (III), then the following statements also hold:
     \begin{enumerate}\setlength\itemsep{0.5em}
\item [\textup{(e)}] There exists some constant $R\coloneq \min\bigparen{\frac{1}{2}(\frac{\tilde{\epsilon}}{\alpha})^{1/\gamma},R_A}$ such that $$\mu(X_n)\leq \alpha(2R)^\gamma\quad \text{ for all }n\in \N,$$
     
    \item [\textup{(f)}] Let $\set{y_n}\inN\subseteq X$ be any sequence such that $y_n\in X_n$ for all $n\in \N$. Then we have  
    $$\#(B_{\frac R2}(y_n)\cap \{y_k \,|\, k\in \N\})\leq (C_{R_A})^5\quad \text{for all $n\in \N$},$$
    where $C_{R_A}$ is the doubling constant.
\end{enumerate}
    
\end{lemma}
\begin{proof} Let $0<\epsilon<1.$ Let $\set{X_n}\inN$ be a partition of $X$ consisting of measurable subsets that satisfies statement $(a)$. We then further partition $X_n$ as follows.  For each $n\in \N$ we partition $H$ into countably many subsets $\set{H_{n,j}}_{j\in\N}$ such that the diameter of $H_{n,j}$ is less than $\Le D^{-1/2}6^{-1}2^{-n}\epsilon$ for all $j\in\N$. Then for each $n\in\N$ we pick $J(n)\in \N$ large enough that 
\begin{equation}\label{bi2}
\sum_{n\in\N}\sum_{j=J(n)+1}^{\infty}\mu\bigparen{X_n\cap \Psi^{-1}(H_{n,j})}<\frac{\epsilon}{3D}.
\end{equation}
 Next, we define 
$$X_{n,j}=X_n\cap\Psi^{-1}(H_{n,j}),\quad \text{for all }1\leq j\leq J(n),n\in\N.$$ For a fixed $n$ and a fixed $j$, there exists a sequence of natural numbers $\set{\ell_{n,j,k}}_{k\in\N}$ such that $\mu(X_{n,j})=\sum_{k=1}^\infty 2^{-\ell_{n,j,k}}.$ Since $\mu$ is non-atomic, for each $k\in \N$ there exists a subset $X_{n,j,k}$ of $X_{n,j}$ with $\mu(X_{n,j,k})=2^{-\ell_{n,j,k}}.$
Let $K(n,j)\in\N$ be large enough that 
\begin{equation}\label{bi3}
\sum_{k=K(n,j)+1}^\infty 2^{-\ell_{n,j,k}}<(3D)^{-1}2^{-j-n}\epsilon\qquad \text{for all }1\leq j\leq  J(n),\,n\in \N.
\end{equation}
Let 
\[
X'=\bigcup_{n\in\N}\bigcup_{j=J(n)+1}^\infty 
\bigparen{X_n\cap \Psi^{-1}(H_{n,j})}\quad\text{and}\quad
X''=\bigcup_{n\in\N}\bigcup_{j=1}^{J(n)}\bigcup_{k=K(n,j)+1}^\infty X_{n,j,k}.
\]
For notational convenience, for any measurable subset $A\subseteq X$ we define the operator $S_A$ by $$S_A(f)=\int_A\ip{f}{\Psi(t)}\Psi(t)\,d\mu(t)\quad \text{for all }f\in H.$$
By \eqref{bi2} we have
\[
\norm{S_{X'}}= \sup_{||f||=1} \int_{X'} |\ip{f}{\Psi(t)}|^2 \, d\mu(t) \le D\mu(X')=\frac{\epsilon}{3}.
\]
Likewise, by \eqref{bi3}
\[
\norm{S_{X''}}= \sup_{||f||=1} \int_{X''} |\ip{f}{\Psi(t)}|^2 \, d\mu(t) \le D\mu(X'')
\le D \sum_{n\in \N} \sum_{j=1}^{J(n)} (3D)^{-1}2^{-j-n}\epsilon
<\frac{\epsilon}{3}.
\]Now we pick one point $x_{n,j,k}$ from $X_{n,j,k}$ satisfying $\norm{\Psi(x_{n,j,k})}^2\leq D .$ Therefore, for each $n,j,k$ and $f\in H$ with $||f||=1$ we have 
\begin{align*}
    \begin{split}
       \ip{(S_{X_{n,j,k}}-2^{-\ell_{n,j,k}}T_{\Psi(x_{n,j,k})})f}{f}
       &= \int_{X_{n,j,k}}\bigip{\ip{f}{\Psi(t)}\Psi(t)-\ip{f}{\Psi(x_{n,j,k})}\Psi(x_{n,j,k})}{f}\,d\mu(t)
       \\
       &=\int_{X_{n,j,k}}(|\ip{f}{\Psi(t)}|^2-|\ip{f}{\Psi(x_{n,j,k})}|^2)\,d\mu(t)
    \end{split}
\end{align*}
By the triangle inequality, the Cauchy–Bunyakovsky–Schwarz inequality, and the bound on diameter of $H_{n,j}$ yields 
\[
\begin{aligned}
\bigg||\ip{f}{\Psi(t)}|^2-|\ip{f}{\Psi(x_{n,j,k})}|^2\bigg|& \le 
(|\ip{f}{\Psi(t)}|+|\ip{f}{\Psi(x_{n,j,k})}|)|\ip{f}{\Psi(t)-\Psi(x_{n,j,k})}|
\\
&\le 2D^{1/2} D^{-1/2}6^{-1}2^{-n}\epsilon\le 3^{-1} 2^{-n} \epsilon.
\end{aligned}
\]
Hence, 
$$ 
\bignorm{S_{X_{n,j,k}}-2^{-\ell_{n,j,k}}T_{\Psi(x_{n,j,k})}}
\le \sup_{||f||=1}|\ip{(S_{X_{n,j,k}}-2^{-\ell_{n,j,k}}T_{\Psi(x_{n,j,k})})f}{f}|
\Le 2^{-\ell_{n,j,k}+1}6^{-1}2^{-n}\epsilon.$$
It follows that 
\begin{align*}
\begin{split}
\bignorm{S_\Psi-\sum_{n\in \N}\sum_{j=1}^{J(n)}\sum_{k=1}^{K(n,j)}2^{-\ell_{n,j,k}}T_{\Psi(x_{n,j,k})}}&\Le \bignorm{S_{X\setminus (X'\cup X'')}-\sum_{n\in \N}\sum_{j=1}^{J(n)}\sum_{k=1}^{K(n,j)}2^{-\ell_{n,j,k}}T_{\Psi(x_{n,j,k})}}+\frac{2\epsilon}{3}\\
&\Le \sum_{n\in \N}\sum_{j=1}^{J(n)}\sum_{k=1}^{K(n,j)}\norm{S_{X_{n,j,k}}-2^{-\ell_{n,j,k}}T_{\Psi(x_{n,j,k})}} +\frac{2\epsilon}{3}\\
&< \frac{2\epsilon}{3}+ \sum_{n\in \N}\sum_{j=1}^{J(n)}\sum_{k=1}^{K(n,j)}2^{-\ell_{n,j,k}+1}6^{-1}2^{-n}\epsilon<\epsilon.
\end{split}
\end{align*}

Finally, let $L_1=0$ and $L_i=\sum_{n=1}^{i-1}\sum_{j=1}^{J(n)}\sum_{k=1}^{K(n,j)}1$ for $i\geq 2.$ Statement (c) then follows by letting $\set{x_k}_{k\in\N}$ be a sequence in $X$ such that for each $n\geq 1$ we have $$\set{x_k}_{k=L_n+1}^{L_{n+1}}= \set{x_{n,j,k}\,|\,1\leq j\leq J(i), 1\leq k\leq K(j,n)}.$$

If $X$ is a metric space equipped with metric $d$ and $(X,\mu)$ satisfies assumptions (II) and (III), then by Lemma \ref{partition_cardinality_lemma}, we can construct a partition $\set{X_n}\inN$ that satisfies statement (a), (e) and (f). Using the same argument, we see that this partition also satisfies statements (b)--(d).
\end{proof}

The sparsity of the sampled sequence obtained in Lemma \ref{sampling_lemma}
allows us to further extract a frame from the sampled sequence in Theorem \ref{frame_distinct_elements} by using an appropriate pairing, partitioning and binary selection procedure. We first establish a variant of a sampling lemma \cite[Lemma 7.3]{Bo24}. By imposing stronger assumptions on weights Lemma \ref{conditions_sampling_function_injective} guarantees that the sampling function takes at most one value in each block of a partition. 

\begin{lemma}
    \label{conditions_sampling_function_injective}
    Let $C>0$ be the absolute constant from Theorem \ref{binary_selector_theorem}.
    Let $\{I_1,\ldots, I_K\}$ be a partition of a finite set $I$  and $\delta>0$. 
    Suppose that:
    \begin{enumerate}[(i)]
        \item 
        $\set{T_i}_{i\in I}$ is a family of positive semi-definite trace-class operators on $H$ with $\text{tr}(T_i)\leq \delta$ for all $i\in I$. 
    \item 
    $\set{\ell_i}_{i\in I}$ is a sequence of natural numbers such that $$T\coloneq \sum_{i\in I}2^{-\ell_i}T_i\leq \textbf{I} \quad \text{and}\quad \sum_{i\in I_k}2^{-\ell_i}\Le \frac{\epsilon^2}{2 C^2\delta}\quad \text{for all }1\leq k\leq K.$$
    \item
    $\Kc$ is a closed subspace of $H$ such that $\gamma\coloneq \text{tr}(P_{\Kc}TP_{\Kc})\leq1$.
      \end{enumerate} 
    Then there exists some subset $I'$ of $I$ such that $$-\epsilon P_{\Kc^\perp}-4\sqrt{\gamma}\textbf{I}\leq \frac{1}{2^\beta}\sum_{n\in I'}T_{n} -T \leq \epsilon P_{\Kc^\perp}+4\sqrt{\gamma}\textbf{I}\quad \text{ and }\quad \#(I'\cap I_k)\leq1\quad \text{for all }1\leq k\leq K,$$
    where $\beta\in\N$ is such that $1<2^{\beta}\frac{\epsilon^2}{C^2\delta}\leq 2$.
\end{lemma}

\begin{proof}
 Let $\beta\in \N$ be such that $1<2^{\beta}\frac{\epsilon^2}{C^2\delta}\leq 2$ and let $r\in \N$ be such that $r\geq \max_{i\in I}\set{\ell_i,\beta}+1$.  For each $i\in I$ we decompose $2^{-\ell_i}T_i$ into $$2^{-\ell_i}T_i= \underbrace{2^{-r}T_i +\cdots+2^{-r}T_i}_{2^{r-\ell_i} \text{ terms}}=\sum_{n=1}^{2^{r-\ell_i}}T_{i,n},
 $$
where $T_{i,n}=2^{-r}T_i$ for $1\leq n\leq 2^{r-\ell_i}$. 
 We define the index set $J_{1}$ as 
$$J_{1}=\set{(i,n)\,|\,1\leq n\leq 2^{r-\ell_i},i\in I},$$
Let $J_2$ be a finite index set such that $\sum_{i\in I} 2^{-\ell_i}+2^{-r} \#J_2 \in \N$. Define ``phantom" operators $\mathcal N_j=0$ for all $j\in J_2$. Let $N=r-\beta$.
 
In the first step we partition $J_{1}$ into sets of size $2$ by grouping $(i,n)$ that share the same index $i$. We also partition $J_2$ into sets of size $2$ in any way.
Let $B_{b}$, $b\in \set{0,1}$, be an arbitrary binary selector of order $1$ corresponding to the partition of $J_1 \cup J_2$. 

We define a sequence of intermediate partitions by sets of size $2$ and a binary selector $B_b$, $b\in \{0,1\}^N$ by a recursive procedure. Suppose that we are given a binary selector
$B_b$, $b\in \{0,1\}^p$ of order $1\le p <N$. Note that each $B_b$ has an even number of elements since
$$2^{-N}\bigparen{\#J_{1}+\#J_{2}}=2^{r-N} \bigg(\sum_{i\in I}2^{-\ell_i}+2^{-r} \#J_2 \bigg)\in \N.$$
Then we choose a partition of each $B_{b}$ by matching elements $(i_1,n_1)\in J_1$ and $(i_2,n_2)\in J_1$ such that $i_1$ and $i_2$ belong to the same set $I_k$ for some $1\leq k\leq K$.
If no more such matching is possible, then we pair the remaining elements in $J_1 \cup J_2$ in an arbitrary way. 
Once we constructed a partition of $B_b$ into sets of size $2$, we let $B_{b0}$ and $B_{b1}$ be any selectors of $B_b$. This recursive procedure yields a binary selector $B_b$, $b\in \{0,1\}^N$ of order $N$. 

Define $B^{(1)}_{b}=B_{b}\cap J_{1}$. Since 
\begin{equation}
\label{cardinality_estimate}
\sum_{i\in I_k}2^{r-\ell_i-N}\Eq 2^\beta\sum_{i\in I_k}2^{-\ell_i}\leq \frac{\epsilon^2}{2\delta C^2}\frac{2\delta C^2}{\epsilon^2}= 1,
\end{equation}
for each $1\le k \le K$, a binary selector $B_b$ of order $N$ contains at most one pair of indices $(i,n)$ originating from $J_{1}$ with $i\in I_k$.  
That is, for each $1\leq k\leq K$ and $b\in \{0,1\}^N$, the set $\set{T_i\,|\, (i,n) \in B^{(1)}_{b}}$ contains at most one $T_i$ with $i\in I_k$.

Now by Theorem \ref{binary_selector_theorem} applied to the operator family $\{T_{i,n}\}_{(i,n) \in J_1}  \cup \{\mathcal N_j\}_{j \in J_2}$ satisfying $\operatorname{tr}(T_{i,n}) \le 2^{-r}\delta$, there exists a binary selector $B_{b}$, $b\in \{0,1\}^N$ of order $N$ such that 
\begin{equation}
\label{binary_selector_estimateII}
\bigg\| 2^N \sum_{(i,n)\in B_{b}^{(1)}}T_{i,n}- T \bigg\|
=
\bigg\|\frac{1}{2^{\beta}}\sum_{(i,n)\in B_{b}^{(1)}}T_{i}-T\bigg\| \Le  C\sqrt{2^{-\beta}\delta}<\epsilon.
\end{equation}
Note that we ignored phantom operators $\mathcal N_j=0$ for all $j\in J_2$.
Since the family $\{B_b^{(1)}\,|\, b\in \{0,1\}^N\}$ forms a partition of $J_1$, we have
\[
\sum_{b\in \{0,1\}^N} \sum_{(i,n)\in B_{b}^{(1)}}2^{-r}T_{i}
=\sum_{b\in \{0,1\}^N} \sum_{(i,n)\in B_{b}^{(1)}}T_{i,n}
= \sum_{(i,n) \in J_1} T_{i,n} = \sum_{i\in I} 2^{-\ell_i} T_i=T.
\]
Thus,
\[
\sum_{b\in \{0,1\}^N} \operatorname{tr}\bigg(P_{\mathcal K} \bigg( \sum_{(i,n)\in B_{b}^{(1)}}2^{-r}T_{i}\bigg) P_{\mathcal K} \bigg)
= \operatorname{tr}(P_{\mathcal K}TP_{\mathcal K}).
\]
Hence, there exists $b\in \{0,1\}^N$ such that
\begin{align}
\begin{split}
  \label{trace_dominate_operator_norm_ineq}  
\text{tr}\Bigparen{P_{\Kc}\Bigparen{\frac{1}{2^{\beta}}\sum_{(i,n)\in B_{b}^{(1)}}T_i}P_{\Kc}}
&=2^N \text{tr}\Bigparen{P_{\Kc}\Bigparen{\sum_{(i,n)\in B_{b}^{(1)}}2^{-r}T_{i}}P_{\Kc}}
\\
&\leq \text{tr}(P_{\Kc} T P_{\Kc})\leq \gamma.
\end{split}
\end{align}
Let $
 I'=\{i \in I\,|\, \exists n \text{ such that }(i,n) \in B^{(1)}_b\} .$
 By the argument following \eqref{cardinality_estimate} we have
 \[
 \#(I' \cap I_k) \le 1 \qquad\text{for all }1 \le k \le K.
 \]
 For notational convenience, we let 
 \[
 S\coloneq \frac{1}{2^\beta}\sum_{(i,n)\in B_{b}^{(1)}}T_i= \frac{1}{2^\beta} \sum_{i\in I'} T_i.
 \]
By \eqref{binary_selector_estimateII} and the assumption that $||T||\le 1$ we have
\begin{align}\label{l2} 
\begin{split}
\bignorm{P_{\Kc^\perp}S P_{\Kc^\perp}}\Le \bignorm{S}<1+\epsilon<2.
\end{split}
\end{align}
By Lemma \ref{elementary_Hilbert_result}, \eqref{trace_dominate_operator_norm_ineq}, and \eqref{l2} we have
\begin{equation}
\label{estimate_1}
-\sqrt{2\gamma}\textbf{I}\Le S-P_{\Kc}S P_{\Kc}-P_{\Kc^\perp}S P_{\Kc^\perp}\Le \sqrt{2\gamma}\textbf{I}.
\end{equation}
Applying the same logic to $T$, we see that 
\begin{equation}
\label{estimate_2}
-\sqrt{\gamma}\textbf{I}\Le T-P_{\Kc}T P_{\Kc}-P_{\Kc^\perp}T P_{\Kc^\perp}\Le \sqrt{\gamma}\textbf{I}.
\end{equation}
Moreover, by equations (\ref{binary_selector_estimateII}) and (\ref{trace_dominate_operator_norm_ineq}), we obtain 
\begin{equation}
\label{estimate_3}
    -\epsilon P_{\Kc^\perp} \Le  P_{\Kc^\perp}\bigparen{S-T}P_{\Kc^\perp} \Le  \epsilon P_{\Kc^\perp},
\end{equation}
and 
\begin{equation}
\label{estimate_4}
    - \gamma P_{\Kc} \Le  P_{\Kc}\bigparen{S-T}P_{\Kc} \Le \gamma P_{\Kc}.
\end{equation}
Combining \eqref{estimate_1}--\eqref{estimate_4} yields  
$$ -(\sqrt{2}+1)\sqrt{\gamma}\textbf{I}-\gamma \textbf{I}-\epsilon P_{\Kc^\perp}\Le S-T\Le (\sqrt{2}+1)\sqrt{\gamma}\textbf{I}+\gamma \textbf{I}+\epsilon P_{\Kc^\perp}. $$
Note that $\gamma\leq 1$. It then follows that
\begin{equation}
\label{estimate_5}
    -\epsilon P_{\Kc^\perp}-4\sqrt{\gamma}\textbf{I}\Le S-T \Le \epsilon P_{\Kc^\perp}+4\sqrt{\gamma}\textbf{I}
\end{equation}
That is, 
$$-\epsilon P_{\Kc^\perp}-4\sqrt{\gamma}\textbf{I}\leq \frac{1}{2^\beta}\sum_{n\in I'}T_{n} -T \leq P_{\Kc^\perp}+4\sqrt{\gamma}\textbf{I},$$ and  $\#(I'\cap I_k)\leq 1$ for all $1\leq k\leq K$.
\end{proof}

Next we show the existence of a sampling set, which is a finite union of uniformly discrete sets, by revamping the proof of discretization result for continuous frames \cite[Theorem 7.1]{Bo24}.

\begin{theorem}
\label{frame_distinct_elements}
  Let $\Psi\colon X\rightarrow H$ be a continuous frame with an upper frame bound $B$. Assume that $\norm{\Psi(t)}^2\Le D$ $\mu$-almost everywhere and $\mu(X)=\infty$. Then for any $0<\epsilon<1$ that is small enough there exists some sequence of distinct elements $\set{x_n}\inN\subseteq X$ such that the following statements hold:
    \begin{enumerate}\setlength\itemsep{0.5em}
  \item [\textup{(a)}]  $\norm{\Psi(x_n)}^2\Le D$ for all $n\in \N,$
   \item [\textup{(b)}] There exists some positive integer $\beta$ with $2^{-\beta}<\frac{\epsilon^2}{16BC^2D}\leq 2^{-\beta+1}$ such that 
   \begin{equation}\label{fre}
   \Bignorm{S_\Psi-2^{-\beta}\sum_{n\in\N}T_{\Psi(x_n)}}<\epsilon.
   \end{equation}
  
\end{enumerate}
Furthermore, if $(X,d,\mu)$ is a metric measure space satisfying assumptions (II) and (III), then the following statement also holds:
    \begin{enumerate}\setlength\itemsep{0.5em}
    \item [\textup{(c)}] There exists some constant $R\coloneq\min\bigparen{\frac{1}{2}\bigparen{\frac{\epsilon^2}{32\alpha B C^2D}}^{\frac{1}{\gamma}}, R_A}$ such that  
    $$\#\bigparen{B_{\frac R2}(x_n)\cap \{x_k\,|\, k\in \N\}}\leq 2(C_{R_A})^5\quad\text{for all }n\in\N.$$ 
    \end{enumerate}
\end{theorem}
\begin{proof} Let $\Psi'=\frac{\Psi}{\sqrt{B}}.$ 
By choosing $\epsilon$ sufficiently small, we may assume that $\frac{\epsilon^2}{16BC^2D}<1$.
By Lemma \ref{sampling_lemma}, there exists a partition $\set{X_n}\inN$ of $X$, a sequence $\set{x_n}\inN\subseteq X$ that satisfies the following properties:   
 \begin{enumerate}\setlength\itemsep{0.5em}
    \item [\textup{(i)}] $\mu(X_n)\leq \frac{\epsilon^2}{32BC^2D}$ for all $n\in\N$.
    \item [\textup{(ii)}] $\norm{\Psi'(x_n)}^2\Le \frac{D}{B}$ for all $n\in \N$.
    \item [\textup{(iii)}] There exists an increasing sequence of positive integers $(L_n)\inN$ with $L_1=1$ such that for each $n\in \N$ we have $\set{x_k}_{k=L_n}^{L_{n+1}-1}\subseteq X_n$. 
    \item [\textup{(iv)}] There exists a sequence of natural numbers $(\ell_n)\inN$ with $\sum_{k=L_n}^{L_{n+1}-1}2^{-\ell_k}\leq \mu(X_n)$ for all $n\in\N$ such that 
    \begin{equation}\label{spsi}
    \Bignorm{S_{\Psi'}-\sum_{n\in\N}2^{-\ell_n}T_{\Psi'(x_n)}}<\frac{\epsilon}{4B}.
    \end{equation}
\end{enumerate}
For notational convenience, we denote $T_{\Psi'(x_n)}$ by $T_n$ for all $n\in \N.$ Note that $\text{tr}(T_n)\leq \frac{D}{B}$ for all $n\in \N.$ Then we define $T\colon H\rightarrow H$ by 
    \begin{equation}
    \label{definition_of_T}
    T:=\sumli 2^{-\ell_n}T_n\Le \textbf{I}.
    \end{equation}
Let $K_0=0$, $K_1=1$ and let $H_1=\set{0}$. We then construct a sequence of orthogonal finite dimensional subspaces $\set{H_k}_{k\in \N}$ of $H$ inductively as follows. Assume that $H_k$ and $K_k$ are already defined for some $k\ge 1$. We define $H_{k+1}$ by 
    $$H_{k+1}=\clspan\set{P_{(H_1\oplus H_2\oplus \cdots \oplus H_k)^\perp}T_n(H)\,|\,1\leq n\leq K_k},$$
    and choose $K_{k+1}\in \N$ large enough such that 
    \begin{equation}\label{2es}
    \{x_n\,|\, K_{k}+1 \le n \le K_{k+1}\} \not\subseteq X_j \qquad\text{for any }j\in \N
    \end{equation}
    and 
    \begin{equation}
    \label{choice_of_Kn}
    \text{tr}\bigg(P_{H_1\oplus H_2\oplus \cdots \oplus H_{k                  
    +1}}\bigg(\sum_{n=K_{k+1}+1}^\infty2^{-\ell_n}T_{n}\bigg)P_{H_1\oplus H_2 \oplus \cdots \oplus H_{k+1}}\bigg)\leq \eta_{k+1},
    \end{equation}
    where $\eta_0=0$ and $\eta_k=\min(B^{-2}4^{-{(k+2)}}\epsilon^2,1)$ for $k\geq 1$. 
   Indeed, the bound \eqref{choice_of_Kn} holds for sufficiently large $K_{k+1}$ since the spaces $H_k$ are finite dimensional and the series in \eqref{definition_of_T} converges in the strong operator topology. Hence, the compression of the same series to a finite dimensional subspace converges in the operator norm.
    Let $M_k=(H_{k+1}\oplus H_{k+2})^\perp$. 
    It follows from the construction of $H_k$ and the choice of $K_k$ that the following statements hold:
    \begin{enumerate}
    \setlength\itemsep{0.5em}
        \item [\textup{(A)}] For all $1\leq n\leq K_{k+1}$. Then we have $T_n(H)\subseteq H_1\oplus H_2 \oplus \cdots \oplus H_{k+2}$.
        \item [\textup{(B)}] By statement (A) and equation (\ref{choice_of_Kn}), we see that for each $k\geq 0$, 
\begin{align}
    \begin{split}
        \text{tr}\bigg(P_{M_k}\bigg(\sum_{n=K_k+1}^{K_{k+1}}2^{-\ell_n}T_n\bigg)P_{M_k}\bigg) &\Eq \text{tr}\bigg(P_{H_1\oplus\cdots \oplus H_k}\bigg(\sum_{n=K_k+1}^{K_{k+1}}2^{-\ell_n}T_n\bigg)P_{H_1\oplus\cdots \oplus H_k}\bigg)\\ &\Le \eta_{k}.
    \end{split}
\end{align}
(In particular, $P_{M_k}(\sum_{n=K_k+1}^{K_{k+1}}2^{-\ell_n}T_n)P_{M_k}=0$ when $k=0,1$.)
 \item [\textup{(C)}] Since  $\bigset{2^{-\frac{\ell_n}{2}}\Psi(x_n)}_{n\in \N}$ is a frame, and hence a complete set in $H$,  the space $\{H_k\}_{k\in \N}$ form an orthogonal decomposition of $H$, and we have 
 \begin{equation}
     \sum_{k=0}^\infty P_{H_{k+1}\oplus H_{k+2}}=2\textbf{I}.
     \end{equation}
    \end{enumerate}
    
Next, for each $k\geq 0$ we define $$I_k=\set{n\in\N\,|\,K_k+1\leq n\leq K_{k+1}}.$$
Let $i_{k,1}$ and $i_{k,2}$ be positive integers such that $x_{K_{k}+1}\in X_{i_{k,1}}$ and $x_{K_{k+1}}\in X_{i_{k,2}}$, respectively. We then accordingly partition $I_k$ into subsets 
$$\{ I_k\cap X_j\,|\, i_{k,1} \le j \le i_{k,2}\}.$$ 
Let $\beta \in\N$ be such that $2^{-\beta}<\frac{\epsilon^2}{16BC^2D}\leq 2^{-\beta+1}$.
Applying Lemma \ref{conditions_sampling_function_injective} to the above partition of $I_k$ and a subspace $M_k$, $k\geq 0$, yields a subset $I_k' \subset I_k$ such that  
$$-\frac{\epsilon}{4B} P_{M_k^\perp}-4\sqrt{\gamma_k}\textbf{I}\leq \frac{1}{2^\beta}\sum_{n\in I_k'}T_{n} -\sum_{n=K_k+1}^{K_{k+1}}2^{-\ell_n}T_n \leq \frac{\epsilon}{4B} P_{M_k^\perp}+4\sqrt{\gamma_k}$$
and $\#\bigparen{\set{x_n\,|\,n\in I_k'}\cap X_j}\leq 1$ for all $i_{k,1}\leq j\leq i_{k,2}$ and all $k\geq 0$.
Let $I'=\bigcup_{k=0}^\infty I'_k.$ By property (C), it follows that
$$\Bignorm{ \frac{1}{2^{\beta}}\sum_{n \in I'}T_{n}-T}< \frac{3\epsilon}{4B}.$$
By \eqref{2es} we have
\begin{equation}
\label{maximum_elements_in_Xj}
\#\bigparen{\set{x_n\,|\,n\in I'}\cap X_j}\leq 2 \quad \text{for all }j\in \N.
\end{equation} 
Hence, by the triangle inequality and \eqref{spsi} we obtain 
\begin{equation}
\label{final_operator_norm_estimate}
\Bignorm{\frac{1}{B}\bigparen{S_\Psi-2^{-\beta}\sum_{n\in I'}T_{\Psi(x_n)}}}\Eq \Bignorm{ S_{\Psi'}-\frac{1}{2^{\beta}}\sum_{n\in I'}T_{n}}< \frac{\epsilon}{B}.
\end{equation}
This proves statements (a) and (b).

In addition, if $(X,d,\mu)$ is a metric measure space satisfying assumptions (II) and (III), then a partition $\set{X_n}\inN$ of $X$ from Lemma \ref{sampling_lemma} satisfies the following property for the constant $R\coloneq\min\bigparen{\frac{1}{2}\bigparen{\frac{\epsilon^2}{32\alpha B C^2D}}^{\frac{1}{\gamma}},R_A}$. If $\set{z_n}\inN\subseteq X$ is a sequence with $z_n\in X_n$ for all $n\in \N$, then $$\#\bigparen{B_{\frac R2}(z_n)\cap\{z_k\,|\, k\in\N\}}\leq (C_{R_A})^5\quad \text{for all }n\in \N.$$  
 Statement (c) then follows by (\ref{maximum_elements_in_Xj}). 
\end{proof}

Finally, we use another binary selection procedure to further increase the sparsity of the sampled frame in Theorem \ref{frame_distinct_elements} to obtain the desired uniformly discrete frame.

\begin{theorem}
\label{frame_uniformly_discrete_elements}
Let $\Psi\colon X\rightarrow H$ be a continuous frame with an upper frame bound $B$, where $X$ is a metric measure space satisfying conditions (I)--(III). Assume that $\norm{\Psi(t)}^2\Le D$ $\mu$-almost everywhere. There exists a constant $C_1>0$ depending only on the doubling constant $X$ such that the following holds. For any sufficiently small $\epsilon>0$, there exists some sequence $\set{x_n}_{n\in I}\subseteq X$ for some $I\subseteq\N$ satisfying the following two conditions:
    \begin{enumerate}\setlength\itemsep{0.5em}
    \item [\textup{(a)}]  There exists $\beta\in \N$ with $2^{-\beta}<\frac{\epsilon^2}{C_1BD}\leq 2^{-\beta+1}$ such that 
    \begin{equation}
    \label{main_result_operator_norm_estimate}
    \Bignorm{S_\Psi-2^{2L-\beta}\sum_{n\in\N}T_{\Psi(x_n)}}<\epsilon,
    \end{equation}
    for some positive integer $1\le L \le 2(C_{R_A})^5$, 
    \item [\textup{(b)}] 
    There exists some constant $r\coloneq\min\bigparen{\frac{1}{8}\bigparen{\frac{\epsilon^2}{2\alpha C_1 B D}}^{\frac{1}{\gamma}},\frac{R_A}{4}}$ such that $\inf\limits_{n\neq m} d(x_n,x_m)\geq r.$
    \end{enumerate}
\end{theorem}
\begin{proof}
Let $0<\epsilon<\min(1,B)$ be small enough and let $C_1=4C^2\cdot 2^{4(C_{R_A})^5}$, where $C$ is the absolute constant as in Theorem \ref{binary_selector_theorem}. 
Let $\beta\in\N$ be such that $2^{-\beta}<\frac{\epsilon^2}{C_1BD}\leq 2^{-\beta+1}$ and let $$r\coloneq\min\Bigparen{\frac{1}{8}\Bigparen{\frac{\epsilon^2}{2\alpha C_1 BD}}^{\frac{1}{\gamma}},\frac{R_A}{4}}.$$
By Theorem \ref{frame_distinct_elements}, there exists some sequence $\set{x_n}\inN\subseteq X$ with $\norm{\Psi(x_n)}^2\Le D$ for all $n\in \N$ such that 
\begin{equation}\label{fra}
\bignorm{S_\Psi-2^{-\beta}\sum_{n\in\N}T_{\Psi(x_n)}}<\frac{\epsilon}{2^{4(C_{R_A})^5}}\quad \text{and}\quad \#\bigparen{B_{2r}(x_n)\cap\{x_k\,|\,k\in\N\}}\leq 2(C_{R_A})^5,
\end{equation}
for all $n\in\N.$ 
    Let $T\colon H\rightarrow H$ be the frame operator of $\Psi':=\frac{\Psi}{\sqrt{2^{\beta+1}B}}$. Note that $T=\frac{1}{2B}(2^{-\beta}\sum_{n\in\N}T_{\Psi(x_n)})$ and $\text{tr}(T_{\Psi'(x_n)})= \frac{D}{2^{\beta+1}B}\leq \frac{\epsilon^2}{2C_1B^2}$.
    By \eqref{fre} we have
    \begin{equation}
    T:=\sum_{n\in\N} T_{\Psi'(x_n)}\Le \frac{B+\epsilon}{2B}<\textbf{I}.
    \end{equation}
    We will iterate the partition and binary selection process to the index set $\N$ until we obtain the desired sequence. We first partition $\N$ by paring $n$ and $m$ if $d(x_n,x_m)<2r$, and repeat this process until no further pair meets this criterion. We then partition $\N$ into the following two sets 
$$J_{1,1}= \set{n\in \N\,|\, x_n\text{ is paired with some $x_m$ with }d(x_n,x_m)<2r},$$
and    
    $$J_{1,2}= \set{n\in \N\,|\, x_n\text{ is left unpaired}}.$$
If $n,m \in J_{1,1}$ are paired, then we pair the corresponding operators $T_{\Psi'(x_n)}$ and $T_{\Psi'(x_m)}$. For each unpaired $n\in J_{1,2}$, we pair the operator $T_{\Psi'(x_n)}$ with a zero operator $\mathcal N_n$ on $H$. This yields a partition of the disjoint union of $\N$ and $J_{1,2}$ into subsets of size two. We then apply Theorem \ref{binary_selector_theorem} to extract binary selectors of order $1$, $B_b$, $b\in \set{0,1}.$ Hence, there exists a subset $N_1$ of $\N$ such that $$\bignorm{2\sum_{n\in N_1}T_{\Psi'(x_n)}-T}<C\sqrt{\frac{2\epsilon^2}{2C_1B^2}}=\frac{\epsilon}{2B\cdot 2^{4(C_{R_A})^5}}.$$
In particular, we have 
$$\#(B_{2r}(x_n)\cap\set{x_k| k\in N_1})\leq 2(C_{R_A})^5-1 \quad \text{for all }n\in J_{1,1}\cap N_1.$$
Next, we perform another binary selection process on $N_1$ as follows. Note that for each $n\in J_{1,2}\cap N_1$, a point $x_n$ is the only element in $B_{2r}(x_n)\cap \set{x_k}_{k\in \N}$ that was left unpaired in the first binary selection process. It follows that $B_r(x_n)\cap B_r(x_m)=\emptyset$ for all distinct $n,m\in N_1\cap J_{1,2}.$ Otherwise, we would have $x_m\in B_{2r}(x_n)$. That is, the distance between $x_n$ and $x_m$ is less than $2r$, but both of them were left unpaired, which is a contradiction. Next, we define 
$$J^{(1)}_{1,2}=\set{n\in J_{1,2}\cap N_1\,|\,B_{r}(x_n)\cap \{x_m | m\in N_1\}\neq \{x_n\}}.$$
Note that for those $n\in N_1 \setminus J^{(1)}_{1,2}$ we either have 
$$\#\bigparen{B_{r}(x_n)\cap \{x_m\,|\, m\in N_1\}}= \#\set{x_n}=1$$ or 
$$\#\bigparen{B_{2r}(x_n)\cap \{x_m \,|\,m\in N_1\}}\leq 2(C_{R_A})^5-1.$$ 
We now partition $N_1$ into sets of size $2$ by first pairing $n\in J_{1,2}^{(1)}$ with $m\in N_1$ if $d(x_n,x_m)<r.$ We then randomly pair the remaining indexes. Likewise, for those $n$ left unpaired, we pair the operator $T_{\Psi'(x_n)}$ with the zero operator $\Nc_{n}$ on $H$. By Theorem \ref{binary_selector_theorem}, there exists some subset $N_2\subseteq N_1$ such that  
$$\bignorm{2\sum_{n\in N_2}T_{\Psi'(x_n)}-\sum_{n\in N_1}T_{\Psi'(x_n)}}<\frac{\epsilon}{2B\cdot 2^{4(C_{R_A})^5}}.$$
Moreover, for each $n\in N_2$ we either have $$\#\bigparen{B_{r}(x_n)\cap\{x_m\,|\, m\in N_1\}}= \#\set{x_n}=1$$ or $$\#\bigparen{B_{2r}(x_n)\cap \{x_m\,|\, m\in N_1\}}\leq 2(C_{R_A})^5-1.$$
We define these two binary selection procedures as one cycle. Assume that $N_{2k}$ is the set obtained after performing this cycle $k$ times for some $k\geq 1$. We then iterate $(k+1)$-th cycle of binary selection procedures on $N_{2k}$, where we define $J_{k+1,1}$, $J_{k+1,2}$, $J_{k+1,2}^{(1)}$ similarly and then obtain $N_{2k+1}$ and $N_{2k+2}$ by Theorem \ref{binary_selector_theorem}. In particular, for each $k\in \N$ we have 
\begin{equation*}
\bignorm{2\sum_{n\in N_{2k+1}}T_{\Psi'(x_n)}-\sum_{n\in N_{2k}}T_{\Psi'(x_n)}}<\frac{\epsilon}{2B\cdot 2^{4(C_{R_A})^5}}, 
\end{equation*}
and 
\begin{equation*}
\bignorm{2\sum_{n\in N_{2k+2}}T_{\Psi'(x_n)}-\sum_{n\in N_{2k+1}}T_{\Psi'(x_n)}}<\frac{\epsilon}{2B\cdot 2^{4(C_{R_A})^5}}.
\end{equation*}
Moreover, we have either $$\#\bigparen{B_{2r}(x_n)\cap \set{x_m}_{m\in N_{2k+1}}}\leq 2(C_{R_A})^5-(k+1)\quad \text{or} \quad \#\bigparen{B_{r}(x_n)\cap \set{x_m}_{m\in N_{2k+1}}}=1$$ for all $n\in N_{2k+2}.$ As a result, it requires at most $2(C_{R_A})^5-1$ iterations of this cycle to obtain a uniformly discrete subset of $\set{x_n}\inN.$ 

Suppose that $\set{x_n}_{n\in N_{2L}}$ is the desired uniformly discrete subset for some $1\leq L\leq 2(C_{R_A})^5-1.$ That is, $d(x_n,x_m)\geq r$ for all distinct $n,m\in N_{2L}$. By telescoping we have
\[\bigg\|2^{2L}\sum_{n\in N_{2L}}T_{\Psi'(x_n)}-\sum_{n\in \N}T_{\Psi'(x_n)}\bigg\| \leq \bigg(\sum_{k=0}^{2L-1} 2^k\bigg)\frac{\epsilon}{2B\cdot 2^{4(C_{R_A})^5}}.
\]
Hence, 
$$
\bignorm{2^{2L-\beta}\sum_{n\in N_{2L}}T_{\Psi(x_n)}-2^{-\beta}\sum_{n\in \N}T_{\Psi(x_n)}} \leq \Bigparen{\sum_{k=0}^{2L-1} 2^{k}}\frac{\epsilon}{2^{4(C_{R_A})^5}}.$$
Since $L\leq 2(C_{R_A})^5$ and $C_{R_A}\geq 1$, it follows by \eqref{fra} that 
\begin{align*}
\begin{split}
\Bignorm{S_\Psi-2^{2L-\beta}\sum_{n\in N_{2L}}T_{\Psi(x_n)}}
&<\frac{\epsilon}{2^{4(C_{R_A})^5}}+
\Bigparen{\sum_{k=0}^{2L-1} 2^{k}}\frac{\epsilon}{2^{4(C_{R_A})^5}} 
\\
&<\Bigparen{1+\sum_{k=0}^{2L-1} 2^{k}}\frac{\epsilon}{2^{2L}}=\epsilon.
\qedhere
\end{split}
\end{align*}
\end{proof}
\begin{remark}
Assume $\set{x_n}\inN$ is a sequence that satisfies the estimate (\ref{main_result_operator_norm_estimate}). Then we have $$(A-\epsilon)\textbf{I}\leq 2^{2L-\beta}\sum_{n\in\N} T_{\Psi(x_n)}\leq (B+\epsilon)\textbf{I}.$$
It follows that $\set{\Psi(x_n)}\inN$ is a frame with frame bounds $2^{\beta-2L}(A-\epsilon)$ and $2^{\beta-2L}(B+\epsilon)$. Hence, Theorem \ref{frame_uniformly_discrete_elements} implies that one can sample a continuous frame by uniformly discrete sequence to obtain a frame with nearly the same frame ratio.   
\end{remark}

\section{Applications}
\label{application}
We conclude this paper by presenting several applications of Theorem \ref{frame_uniformly_discrete_elements} to Gabor frames, wavelet frames, and frames of exponentials.

\subsection{Gabor systems}
The first  continuous frame of particular interest to us is the \emph{short-time Fourier transform}. Since its introduction in the early 1980s by Feichtinger and subsequent development in joint work with Gr\"{o}chenig, short-time Fourier transform and \emph{modulation spaces} have been recognized as the fundamental tool and function spaces for the study of time-frequency analysis. We refer to \cite{Fei81}, \cite{Fei06} and \cite{Gro01} and the references therein for more history context and background knowledge on this topic.    
Let $d\geq1$ and $g\in L^2(\R^d)$ be a nonzero function. The short-time 
Fourier transform (with window function $g$) is the map $\Psi_g\colon \R^{2d}\rightarrow L^2(\R^d)$ defined by $$\Psi_g(a,b)(x)=e^{-2\pi ibx}g(x-a), \quad \text{for all }a,b\in \R^{d}.$$
It is known that $\Psi_g$ is a continuous tight frame with a frame bound $\norm{g}^2_{L^2(\R^d)}$ (for example, see \cite[Section 3.2]{Gro01}). The Euclidean space $\R^{2d}$ equipped the usual Lebesgue measure clearly satisfies conditions (II) and (III). In particular, $\mu(B_{2r}(x))\leq 2^{2d}\mu(B_{r}(x))$ for all $r>0$ in this setting. 
As a result, we obtain the following Theorem.


\begin{theorem} 
\label{application_I_Gabor}
Let $d\geq1$ and let $g\in L^2(\R^d)$ be a nonzero function. For any sufficiently small $\epsilon>0$, there exists a uniformly discrete subset $\Lambda=\set{(a_n,b_n)}\inN\subseteq \R^{2d}$ such that the following statements hold.
\begin{enumerate}
\setlength\itemsep{0.3em}
    \item [\textup{(a)}] $\Gc{(g,\Lambda)}$ is a frame for $L^2(\R^d)$ with frame bounds $C_1 \norm{g}^2_{L^2(\R^d)} \epsilon^{-2} (1 \pm\epsilon)$
     \item [\textup{(b)}] $\inf\limits_{n\neq m}\bignorm{(a_n,b_n)-(a_m,b_m)}\geq {C_2} \epsilon^{1/d}$,
\end{enumerate}
 for some constants $C_1, C_2>0$ depending only the dimension $d$. \qeddef
\end{theorem}

\subsection{Wavelet systems} The continuous wavelet transform associated with a nonzero function $\psi\in L^2(\R)$ is the mapping $\pi\colon \R\rtimes \R^+\rightarrow L^2(\R)$ defined by $$\pi(b,a)(\psi)(x)=a^{1/2}\psi(ax-b).$$
Unlike short-time Fourier transform, an additional admissibility condition (\ref{admissibility_condition}) must be made to make the continuous wavelet transform associated with the affine group $\R\rtimes \R^+$ a continuous tight frame. By identifying $\R\rtimes \R^+$ as the hyperbolic plane, one can define the distance function $d$ on $\R\rtimes \R^+$ by $$d\bigparen{(b_1,a_1),(b_2,a_2)}=\text{arccosh}\bigg(1+\frac{|(b_1,a_1)-(b_2,a_2)|^2}{2a_1a_2}\bigg).$$
With the left-invariant Haar measure associated with $\R\rtimes \R^+$ defined by $d\mu(b,a)=a^{-2}dbda$, we obtain the following identity for the measure of balls with radius $r$ (with respect to the hyperbolic metric)
\begin{equation}
\label{measure_ball_hyperbolic}
\mu(B_{r}(b,a))=2\pi (\cosh(r)-1)\quad \text{for all }(b,a)\in \R\rtimes \R^+~\text{and all }r>0.
\end{equation}
For the derivation of (\ref{measure_ball_hyperbolic}), as well as references for hyperbolic geometry, we refer to \cite{Be83} and \cite{Ka92}.
As a result, all conditions (I)--(III) are satisfied. Indeed, \eqref{measure_ball_hyperbolic} implies that $\mu(B_{r}(b,a))$ behaves asymptotically as $\pi r^2$ when $r\to 0$. We accordingly obtain the following theorem. 
\begin{theorem}
  Let $\psi\in L^2(\R)$ be a nonzero function satisfying the admissibility condition 
  \[
  B= 
\int_{(-\infty,0)}\frac{|\widehat{\psi}(\xi)|^2}{|\xi|}\,d\xi
= \int_{(0,\infty)}\frac{|\widehat{\psi}(\xi)|^2}{|\xi|}\,d\xi<\infty.
\]
Then for any sufficiently small $\epsilon>0$, there exists a uniformly discrete subset $\Gamma=\set{(b_n,a_n)}\inN\subseteq \R\rtimes \R^+$ such that the following statements hold.
\begin{enumerate}
\setlength\itemsep{0.3em}
    \item [\textup{(a)}] $\Wc(\psi,\Gamma)$ is a frame for $L^2(\R)$ with frame bounds $C_1 \norm{\psi}_{L^2(\R^d)}^2 \epsilon^{-2} (1\pm \epsilon)$,
     \item [\textup{(b)}] $\inf\limits_{n\neq m}\bignorm{(a_n,b_n)-(a_m,b_m)}\geq C_2  B^{1/2}\norm{\psi}^{-1}_{L^2(\R^d)} \epsilon$,
\end{enumerate}
 for some absolute constants $C_1, C_2>0$. \qeddef
\end{theorem}

The versatility of Theorem \ref{frame_uniformly_discrete_elements} makes it applicable not only for one dimensional, but also for higher dimensional continuous wavelets in $L^2(\R^d)$. However, we will not explore this direction here.

\subsection{Exponential systems}
Although Fourier transform on $L^2(\R^d)$ is not a continuous frame, the restriction of Fourier transform to $L^2(S)$ is a Parseval frame for any subset $S$ of $\R^d$ with finite, positive  measure. Therefore, Theorem \ref{frame_uniformly_discrete_elements} recovers the following result, which was shown in \cite[Corollary 8.5]{Bo24} using different argument.

\begin{theorem} 
\label{application_III_exponential}
Let $d\geq 1$ and let $S\subseteq \R^d$ be a set of finite, positive measure. For any sufficiently small $\epsilon>0$, there exists a uniformly discrete subset $\set{\lambda_n}\inN$ of $\R^d$ such that the following statements hold:
\begin{enumerate}
\setlength\itemsep{0.3em}
    \item [\textup{(a)}] $\set{e^{2\pi i\lambda_nx}}\inN$ is a frame for $L^2(S)$ with frame bounds $C_1 |S| \epsilon^{-2}(1\pm\epsilon)$,
     \item [\textup{(b)}]  $\inf\limits_{n\neq m}|\lambda_n-\lambda_m|\geq C_2 \bigparen{\frac{\epsilon^2}{|S|}}^{1/d}$,
\end{enumerate}
 for some constants $C_1, C_2>0$ depending only the dimension $d$. \qeddef
\end{theorem}

\subsection{Reproducing kernels associated with certain spectral subspaces}
Another application arises from the reproducing kernels associated with separable, infinite-dimensional reproducing kernel Hilbert spaces contained in $L^2(X,\mu)$. Here $L^2(X,\mu)$ denotes the Hilbert space of square $\mu$-integrable functions defined on $X$. 
We say that a Hilbert space $H\subseteq L^2(X,\mu)$
is a \emph{reproducing kernel Hilbert space} if for each $x\in X$ there exists some function $K_x\in H$ such that
for all $f\in H$ we have $$f(x)=\ip{f}{K_x}=\int_Xf(t)\overline{K_x}(t)\,d \mu(t)\quad \text{for all }x\in X.$$
The \emph{reproducing kernel} associated with $H$ is the function $K\colon X\times X\rightarrow \C$ defined by 
$$K(x,y)=\ip{K_y}{K_x}$$
In particular, the function $\Psi\colon X\rightarrow H$ defined $\Psi(x)(\cdot)=K_x(\cdot)=\overline{K(x,\cdot)}$ is a continuous Parseval frame for $H.$ 
Consequently, under the conditions of Theorem \ref{frame_uniformly_discrete_elements}, one can always construct uniformly discrete frames for a reproducing kernel Hilbert space with frame bounds as close as desired by sampling its associated reproducing kernel. For example, for any $A>0$ the classical Paley-Wiener Space $$PW_A=\set{f\in L^2(\R)\,|\,\text{supp}(\widehat{f})\subseteq[-A,A]}$$
is a reproducing kernel Hilbert space with the reproducing kernel $K(x,y)=\frac{\sin( A(x-y))}{\pi(x-y)}$. Here $\widehat{f}$ denotes the Fourier transform of $f\in L^2(\R^d).$ Euclidean spaces equipped with the usual Euclidean metric and the Lebesgue measure clearly satisfies assumptions (I)--(III). One can generalize the notion of Paley-Wiener space by viewing the Paley-Wiener space as a spectral subspace associated to an specific positive semi-definite, self-adjoint operator, see \cite{GK24} and \cite{PZ09}.
For every $s\geq 0$, the Sobolev space $W^s_2(\R^d)$ is defined by $$W^s_2(\R^d)=\bigset{f\in L^2(\R^d)\,|\,\norm{f}^2_{W^s_2(\R^d)}=\int_{\R^d} |\widehat{f}(\xi)|^2(1+|\xi|^2)^s<\infty}.$$ 
Let $a=(a_{ij})_{1\leq i,j\leq d}$ be a $d\times d$ positive definite matrix (i.e. there exists some $C>0$ such that $(a(x)\xi)\cdot\xi\geq C\norm{\xi}^2$ for all $x,\xi\in \R^d$) with bounded infinitely differentiable entries $a_{ij}$. The differential operator $H_a$ (with symbol $a$) defined by 
\begin{equation}
\label{differential_operator}
H_af=-\sum_{i,j=1}^d\partial_ia_{ij}\partial_jf,\quad f\in W^2_2(\R^d)
\end{equation}
is a positive semi-definite, uniformly elliptic self-adjoint operator on $\R^d.$ For any $A>0$, we define the \emph{spectral projection} $\chi_{A,H_a}$ associated with $H_a$ on the interval $[0,A]$ by $$\chi_{A,H_a}(f)=\Fc^{-1}\bigparen{\chi_{[0,A]}\cdot\bigparen{\Fc H_a(f)}},$$where $\Fc$ denotes the usual Fourier transform. The \emph{spectral subspaces} $PW_A(H_a)$ associated with $H_a$ is then defined by $PW_A(H_a)=\chi_{A,H_a}(L^2(\R^d))$. By \cite[Proposition 2.2]{GK24}, $PW_A(H_a)$ is a reproducing kernel Hilbert space in $L^2(\R^d)$. Furthermore, the underlying space $\R^d$ satisfies conditions (I)--(III). We accordingly obtain the following theorem.

\begin{theorem} \label{application_spectral_subspace}
Fix $A>0.$
   Let $a=(a_{ij})_{1\leq i,j\leq d}$ be a $d\times d$ positive definite matrix with bounded infinitely differentiable entries $a_{ij}$. 
   Let $K$ be the reproducing kernel associated with $PW_A(H_a)$, where $H_a$ is the differential operator with symbol $a$ defined in Equation (\ref{differential_operator}).
Then for any sufficiently small $\epsilon>0$, there exists a uniformly discrete subset $\set{\lambda_n}\inN\subseteq \R^d$ such that the following statements hold:
\begin{enumerate}
\setlength\itemsep{0.3em}
    \item [\textup{(a)}] $\set{K_{\lambda_n}}\inN$ is a frame for $PW_A(H_a)$ with frame bounds  $C_1 \tilde C \epsilon^{-2} (1\pm \epsilon)$ ,
     \item [\textup{(b)}] $\inf\limits_{n\neq m}\norm{\lambda_n-\lambda_m}\geq C_2 (\epsilon^2/\tilde C)^{1/d}$, \qeddef
\end{enumerate}
for some constants $C_1, C_2>0$ depending only the dimension $d$ and some constant $\widetilde{C}$ depending on the reproducing kernel $K$.\qeddef
\end{theorem}
We remark that the constant $\widetilde{C}$ in Theorem \ref{application_spectral_subspace} can be found explicitly by the elliptic regularity of $H_a$ and Bernstein inequality (see \cite[Lemma 2.1 and Propostion 2.2]{GK24}).

\end{document}